\documentclass[reqno,10pt]{amsart}
\usepackage[dvipsnames,usenames]{color} 
\usepackage[toc,page]{appendix}

\usepackage{tikz}  
\usepackage{mathcomp,wasysym}  

\usepackage{hyperref}
\usepackage[mathscr]{euscript}  
         
\usepackage{cite}      
\usepackage{graphicx}  
\usepackage[all]{xy} \xyoption{arc} \xyoption{color}

\usepackage{amsmath} 
\usepackage{amsthm} 
\usepackage{amsfonts}  
\usepackage{amssymb} 

\usepackage{verbatim}

\usepackage{mathrsfs}

\newcommand{\RR}{\mathbb R}

\newcommand{\Zd}{{\mathbb Z^d}}

\newcommand{\Td}{{\mathbb T^d}}
\newcommand{\TT}{{\mathbb T^2}}

\newcommand{\pat}{\partial_t}
\newcommand{\pax}{\partial}

\newcommand{\jeps}{\mathcal{J}_\epsilon*}

\newcommand{\vertiii}[1]{{\left\vert\kern-0.25ex\left\vert\kern-0.25ex\left\vert #1 
    \right\vert\kern-0.25ex\right\vert\kern-0.25ex\right\vert}}

\newcounter{comentcount}
\newcounter{teocount}
\newtheorem{lem}{Lemma}

\newtheorem{teo}[teocount]{Theorem}  
\newtheorem{defi}{Definition}

\newtheorem{remark}{Remark}

\title[$2$D fractional Keller-Segel equation with logistic source]{Suppression of blow up by a logistic source in $2$D  Keller-Segel system with fractional dissipation}

\author[J. Burczak]{Jan Burczak}
\email{jb@impan.pl}
\address{Institute of Mathematics, Polish Academy of Sciences, ul. \'Sniadeckich 8, 00-656 Warsaw, Poland}

\author[R. Granero-Belinch\'{o}n]{Rafael Granero-Belinch\'{o}n}
\email{granero@math.univ-lyon1.fr}
\address{Univ Lyon, Universit\'e Claude Bernard Lyon 1, CNRS UMR 5208, Institut Camille Jordan, 43 blvd. du 11 novembre 1918, F-69622 Villeurbanne cedex, France.}

\begin{document}

\begin{abstract}
We consider a two dimensional parabolic-elliptic Keller-Segel equation with a logistic forcing and a fractional diffusion of order $\alpha$. We obtain existence of global in time regular solution for arbitrary initial data with no size restrictions and $c<\alpha\leq 2$, where $c \in (0,2)$ depends on the equation's parameters. For an even wider range of $\alpha's$, we prove existence of global in time weak solution for general initial data.
\end{abstract}

\maketitle 


\section{Introduction}

Let us consider the following drift-diffusion equation on the two-dimensional torus  $\TT$
\begin{equation}\label{eqDD}
\pat u=- \Lambda^\alpha u+\chi\nabla\cdot(u B(u))+f(u),
\end{equation}
where $\Lambda^\alpha =(-\Delta)^{\alpha/2}$, $B(u)$ is a vector of nonlocal operators, $f$ denotes a real function and $\chi $  is a sensitivity parameter. A concrete choice of $B, f, \chi$ yields one of many so-called \emph{active scalar} equations. The notion \emph{active scalar} refers to the main unknown $u$ being a scalar advected by a vector field depending on the scalar itself, here $B(u)$, sometimes under effects of diffusion $\Lambda^\alpha$ or some other forces $f$. This family of equations appears extensively in applied mathematics. In particular, evolution of some of the most famous two-dimensional active scalars can be seen as a special case of \eqref{eqDD}, including:
\begin{itemize}
\item The two dimensional incompressible Euler / Navier-Stokes equation in its vorticity formulation
$$
\alpha=0 \text{ / } \alpha=2,\quad B(u)=(-\partial_{x_2},\partial_{x_1}) (-\Delta)^{-1}u, \quad f=0, \quad  \chi = 1,
$$ 
describing a flow of inviscid/viscous fluid. \vskip 2mm
\item The Surface Quasi-Geostrophic / Dispersive Surface Quasi-Geostrophic equation
$$
0<\alpha\leq 2,\quad B(u)=(-\partial_{x_2},\partial_{x_1}) \Lambda^{-1}u, \quad f=0 \text{ / } f (u) = R_1 u, \quad  \chi = 1,
$$ 
where $R_i$ is the $i-th$ Riesz transform (see Section \ref{ssec:not} for notation), compare  \cite{constantin1994singular, constantin1999formation, constantin2013long,berselli2002vanishing, caffarelli2010drift, KiselevSQG, Omar, Omar2, KisNaz10, Held_sqg}. It models temperature evolution in certain geophysical considerations (meteorology, oceanography, simplified magnetodynamics). Besides, it is supposed to provide insights into behaviour of three-dimensional flows. \vskip 2mm
\item The Incompressible Porous Medium equation (self-explanatorily, related to flows in porous media)
$$
0<\alpha\leq 2, \quad B(u)=(\partial_{x_2}\partial_{x_1}, -\partial_{x_1}^2) (-\Delta)^{-1}u \quad (= - R^{\perp}R_1u),  \quad f=0, \quad  \chi = 1,
 $$
where
$$
R=(R_1,R_2),\quad R^{\perp}=(-R_2,R_1),
$$  see for instance \cite{castro2009incompressible, Friedlander};

\vskip 2mm
\item The Stokes equation
$$
0<\alpha\leq 2, \quad B(u)=(-\Delta)^{-1}R^{\perp}R_1u, \quad f=0, \quad  \chi = 1.
$$ 
see  \cite{bae2015global}. Note that this equation can be written equivalently as the following system of differential equations
\begin{align*}
\pat u+\nabla\cdot(uv)&=0,\\
-\Delta v+\nabla p&=-(0,u),\\
\nabla\cdot v&=0.
\end{align*}
\vskip 2mm
\item The Magnetogeostrophic equation 
$$
0 \leq \alpha\leq 2,\quad \hat{B}(u)=\frac{\left(\xi_2\xi_3|\xi|^{2}-\xi_1\xi_2^2\xi_3,-\xi_1\xi_3|\xi|^{2}-\xi_2^3\xi_3,\xi_1^2\xi_2^2+\xi_2^4\right)}{|\xi|^{2}\xi_3^2+\xi_2^4}\hat{u},   \quad f=0, \quad  \chi = -1,
$$
where $\hat{\cdot}$ denotes the Fourier transform. It is a simplified model for creation of EarthÕs magnetic field,  compare \cite{Friedlander2, Friedlander3, Friedlander4, Friedlander5, Moffatt}.
\vskip 2mm
\item A class of aggregation equations 
$$
0<\alpha\leq 2, \quad B(u)=\nabla K * u, \quad \chi >0,
$$
where $K$ stands for a nonincreasing (thus attractive) interaction kernel.
For the fractional case with $f=0$, see \cite{biler2009blowup, li2009finite, li2009refined, li2010wellposedness, li2010exploding, BilerWu}. The case $\alpha=2$ and $K = -\Delta^{-1}$  is the classical (parabolic-elliptic) Keller-Segel equation, describing concentration of certain microorganisms. Introduction of damping $ f(u) =u(1-u)$ allows to capture birth and death process, compare model M8 of \cite{Hillen3} and \cite{TelloWinkler}. 
Compare also \cite{Gsemiconductor} for the case of a repulsive kernel related to semiconductor devices. \vskip 2mm
\end{itemize}

Some of the presented active scalars may be reformulated also in higher dimensions. However, the case of two dimensions turns out to be the pivotal one mathematically in certain cases. This involves SQG and aggregation equations. The latter, being our main motivation, is discussed more thoroughly in what follows. Let us only remark that despite our main interest in a Keller-Segel related problems, we expect that our approach may be useful for studying a wide range of damped aggregation equations.

\subsection{The Keller-Segel system}
Our main motivation to study \eqref{eqDD} is the following parabolic-elliptic Keller-Segel system with logistic source
\begin{equation}\label{eqPKS}
\pat u=- \Lambda^\alpha u+\chi \nabla\cdot(u B (u) )+ru(1-u),
\end{equation}
with either
\begin{equation}\label{f:B1}
B(u)=\Lambda^{\beta-1}R(1+\Lambda^{\beta})^{-1}(u), \quad \beta >0
\end{equation}
or 
\begin{equation}\label{f:B}
B (u) = \nabla \Delta^{-1}(u -\langle u \rangle) \qquad ( = \Lambda^{\beta-1}R\Lambda^{-\beta} (u -\langle u \rangle) \text{ with any } \beta >0)
\end{equation}
where $\langle u \rangle$ denotes the spatial mean value $\frac{1}{4 \pi^2} \int_\TT u (t)$.
\vskip 2mm

The classical (doubly parabolic) Keller-Segel system
\begin{equation}\label{cKS}
\begin{aligned}
\pat u &= \Delta u-\chi \nabla\cdot(u \nabla v ) \\
\pat v &= \Delta v + u - \gamma v
\end{aligned}
\end{equation}
serves as a model of \emph{chemotaxis}, \emph{i.e.} a proliferation and a chemically-induced motion of cells, see the pioneering work of  Keller \& Segel \cite{keller1970initiation}, reviews by Blanchet \cite{blanchet2011parabolic} and Hillen \& Painter \cite{Hillen3}. In this interpretation $u\geq0$ is a density of cells and $v$ stands for a density of a chemoattractant. The parameter $\chi>0$ quantifies the sensitivity of organisms to the attracting chemical signal and $\gamma \ge 0$ models their decay.  It is fair to recall that before appearance in mathematical biology, the Keller-Segel equation was introduced by Patlak \cite{patlak1953random} in a context of quantitive chemistry and physics (interestingly, in the Bulletin of Mathematical Biophysics). A biologically justified parabolic-elliptic simplification of \eqref{cKS} consists in rewriting the equation for $v$ as  $0= \Delta v + u - \gamma v$, \emph{i.e.} $v= (\gamma-\Delta)^{-1}(u)$, which gives \eqref{eqPKS} with $\alpha=2$, $r=0$, $\beta=2$ and either \eqref{f:B1} for $\gamma =1$ (since $\Lambda R = - \nabla$) or \eqref{f:B} for $\gamma =0$ (up to lack of subtraction of mean value there, which appears on torus naturally). Thus, one obtains the parabolic-elliptic Keller-Segel model
\begin{equation}\label{cKS2}
\begin{aligned}
\pat u &= \Delta u-\chi \nabla\cdot(u \nabla v ) \\
v&= (\gamma-\Delta)^{-1}(u).
\end{aligned}
\end{equation}
In one space dimension \eqref{cKS2}  admits large-data global in time smooth solutions. In higher dimensions there are small-data or short-time smoothness results, but generally solutions may exhibit finite-time blowups for large data. Here the two-dimensional case ($d=2$) bears a special importance, since the scaling invariant Lebesgue space ${L^\frac{d}{2}}$ (for $\gamma=0$ and full space case), a reasonable choice to investigate smoothness/blowup dichotomy, coincides with the space where the quantity conserved over the evolution by  \eqref{cKS2}  itself lies ($\|u(0)\|_{L^1} = \|u(t)\|_{L^1}$)  and simultaneously with the most natural choice from the perspective of applications, \emph{i.e.} the total mass.
The related literature is abundant, so let us only mention the seminal results by J\"ager \& Luckhaus \cite{jager1992explosions} and  Nagai \cite{nagai1995blow}, the concise note by Dolbeault \& Perthame \cite{Dolbeault2}, where the threshold mass $\frac{8 \pi}{\chi}$ is easy traceable as well as  Blanchet, Carrillo \& Masmoudi \cite{BCM}, focused on the threshold mass case. Interestingly, even in this most classical case a \emph{single} quantity responsible for jointly local existence and blowup criterion is still not fully agreed upon, since for a local-in-time existence one needs to assume more than merely finiteness of the initial mass. Currently, the best candidate seems to be the scaling-invariant Morrey norm, compare Lemari\'e-Rieusset \cite{lemarie2013ks},  Biler, Cie\'slak, Karch \& Zienkiewicz \cite{biler2014local} and its references. 

Remarkably, it has been noted by A. Kiselev \& X. Xu \cite{mixingKS} and J. Bedrossian \& S. He \cite{mixingKS2} that mixing may prevent finite time singularities.  

Let us immediately clarify, in context of the eponymous {supercriticality}, that we simply call a Keller-Segel-related system in two spatial dimensions  \emph{supercritical}, if it involves a weaker dissipation than the classical one $\Delta u$.

As already mentioned, introduction of the logistic  term $r u (1-u)$ to a Keller-Segel system allows to capture cells growth. For instance, the cell-kinetics model M8 of  \cite{Hillen3}  is precisely \eqref{cKS} with added $r u (1-u)$. One of the most striking results for the classical Keller-Segel says that presence of a logistic source prevents a blowup of solutions, compare Tello \& Winkler \cite{TelloWinkler} for the interesting for us parabolic-elliptic case.

\vskip 2mm

Since 1990's, a strong theoretical and empirical evidence has appeared for replacing in Keller-Segel equations the classical diffusion with a fractional one: $\Lambda^\alpha$, $\alpha<2$, in order to model feeding strategies of a wide range of organisms.  For more details, the interested reader may consult our \cite{BG3} with its references.

Because  $ -\Lambda^\alpha u$ provides for $\alpha<2$ a weaker dissipation than the classical one, it is expected that a blowup may easily occur (in particular, for any initial mass, but this is automatic, provided the scaling applies). It is indeed the case for the generic fractional parabolic-elliptic cases, recall \cite{biler2014local, biler2009blowup, li2009finite, li2009refined, li2010wellposedness, li2010exploding, BilerWu} (Naturally there are small-data global regularity results available, based upon data in certain Lebesgue or more involved spaces). On the other hand, though, we know already that an addition of the logistic term prevents blowups in the classical case \eqref{cKS2}.

\subsection{A question and a sketch of our answer}
In context of the presented state-of-the-art it is very natural to pose the following problem
\vskip 2mm
\begin{center}
\emph{
Does the logistic term prevent blowups of solutions to \eqref{eqPKS} with $\alpha<2$?
}
\end{center}
\vskip 2mm
This paper is focused on addressing this question. Briefly, our answer reads at follows:

\vskip 2mm

\emph{
As long as a positive 
$$
\alpha>2 \left( 1- \frac{\chi}{r} \right),
$$
then \eqref{eqPKS} with $B$ given either by \eqref{f:B1} or by \eqref{f:B} has a global in time, classical solution (see Theorem \ref{th:globalstrong}). Moreover, in a wider range of $\alpha$'s, the considered problem has a global-in-time weak solution (see Theorem \ref{th:globalweak}). Namely
\begin{itemize}
\item for every $\alpha>0$, provided   $2 r \geq \chi$,
\item and for every positive  $\alpha >  4 \frac{ \chi ( \frac{\chi - 2r}{\chi - r}) - r}{{2\chi-r}}$  otherwise.
\end{itemize}
}


\vskip 2mm
Our result can be seen as another method (besides mixing as in A. Kiselev \& X. Xu \cite{mixingKS} and J. Bedrossian \& S. He \cite{mixingKS2}) to suppress finite time blow up inherent to the two-dimensional Keller-Segel with large masses.

In relation to our main research question, let us recall our one-dimensional studies in \cite{BG, BG4}. We have also studied the doubly-parabolic problem in \cite{BG2} and provided smoothness for the critical fractional one-dimensional case with no logistic damping, compare \cite{BG3} (see also \cite{AGM}).

\subsection{Plan of the paper} 
The next section contains needed preliminaries and definitions. In Section \ref{sec:3} we provide our main results. Further sections are devoted to their proofs.

\section{Preliminaries}
In what follows, we provide certain formulas in an arbitrary dimension $d$.
\subsection{Notation}\label{ssec:not}
We write $R_j$, $j=1, \dots, d$, for the $j-$th Riesz transform and $\Lambda^s =(-\Delta)^{s/2}$, \emph{i.e.}
\begin{equation}\label{Hfourier}
\widehat{R_j u}(\xi)=-i\frac{\xi_j}{|\xi|}\hat{u}(\xi), 
\end{equation}
\begin{equation}\label{Lfourier}
\widehat{\Lambda^s u}(\xi)=|\xi|^s\hat{u}(\xi),
\end{equation}
where $\hat{\cdot}$ denotes the usual Fourier transform. These operators have the following kernel representation, compare \cite{calderon1954singular}.
\begin{equation}\label{Rkernel}
R_if(x)= r_d \; \text{P.V.}\int_{\Td} f(x-y)\frac{y_i}{|y|^{d+1}}dy+  r_d \; \sum_{k\in \Zd, k\neq 0} \text{P.V.}\int_{\Td}f(x-y)\left(\frac{y_i-2k_i\pi}{|y-2k\pi|^{d+1}}+\frac{2k_i\pi}{|2k\pi|^{d+1}}\right)dy,
\end{equation}
with
\[
r_d = \frac{\Gamma(1+d/2)}{\pi^{\frac{d+1}{2}}}.
\]
\begin{align}\label{Lkernelalpha}
\Lambda^\alpha f(x)&=c_{\alpha, d}\bigg{(}\sum_{k\in \Zd, k\neq 0}\int_{\Td}\frac{f(x)-f(x-\eta)d\eta}{|\eta+2k\pi|^{d+\alpha}} +\text{P.V.}\int_{\Td}\frac{f(x)-f(x-\eta)d\eta}{|\eta|^{d+\alpha}}\bigg{)},
\end{align}
where
$$
c_{\alpha,d} = \frac{2^\alpha\Gamma(\frac{d+\alpha}{2})}{\pi^\frac{d}{2} |\Gamma(-\alpha/2)|}.
$$
Usually, we write 
$$
R=(R_1, \dots, R_d).
$$
Then, observe the following identity
$$
\nabla\cdot R=\Lambda.
$$
Next, we denote by $\mathcal{J}_\epsilon$ the periodic heat kernel at time $t=\epsilon$.

\subsection{Function spaces}
Let us write $\pax^n,$ $n\in\mathbb{Z}^+$, for a generic derivative of order $n$. Then, the fractional $L^p$-based Sobolev spaces (also known as Aronsztajn, Gagliardo or Slobodeckii spaces) $W^{s,p}(\Td)$ are 
$$
W^{s,p} (\Td)=\left\{f\in L^p(\Td) \; | \quad \pax^{\lfloor s\rfloor} f\in L^p(\Td), \frac{|\pax^{\lfloor s\rfloor}f(x)-\pax^{\lfloor s\rfloor}f(y)|}{|x-y|^{\frac{d}{p}+(s-\lfloor s\rfloor)}}\in L^p(\Td\times\Td)\right\},
$$
endowed with the norm
$$
\|f\|_{W^{s,p}}^p=\|f\|_{L^p}^p+\|f\|_{\dot{W}^{s,p}}^p, 
$$
$$
\|f\|_{\dot{W}^{s,p}}^p=\|\pax^{\lfloor s\rfloor} f\|^p_{L^p}+\int_{\Td}\int_{\Td}\frac{|\pax^{\lfloor s\rfloor}f(x)-\pax^{\lfloor s\rfloor}f(y)|^p}{|x-y|^{d+(s-\lfloor s \rfloor)p}}dxdy.
$$
In the case $p=2$, we write $H^s(\Td)=W^{s,2}(\Td)$ for the standard non-homogeneous Sobolev space with its norm
$$
\|f\|_{H^s}^2=\|f\|_{L^2}^2+\|f\|_{\dot{H}^s}^2, \quad \|f\|_{\dot{H}^s}=\|\Lambda^s f\|_{L^2}.
$$ 
\vskip 2mm
Next, for $s\in (0,1)$, let us denote the usual H\"older spaces as follows
$$
C^{s} (\Td)=\left\{f\in C(\Td) \;| \quad \frac{|f(x)-f(y)|}{|x-y|^{s}}\in L^\infty(\Td\times\Td)\right\},
$$
with the norm
$$
\|f\|_{C^{s}}=\|f\|_{L^\infty}+\|f\|_{\dot{C}^{s}},\quad
\|f\|_{\dot{C}^{s}}=\sup_{(x,y)\in\Td\times\Td}\frac{|f(x)-f(y)|}{|x-y|^{s}}.
$$
By $\mathcal{D}$ we mean smooth test functions.  For brevity, the domain dependance of a function space will be generally suppressed. Finally, we will use the standard notation for evolutionary (Bochner) spaces, writing $L^p(0,T; W^{s,p})$ etc.
\subsection{Weak solutions}
We adopt the following standard

\begin{defi}\label{def:weaksol}$u \in L^2(0,T;L^2)$ is a global weak solution of \eqref{eqPKS} emanating from $u_0 \in L^2$ iff for all $T>0$ and $\phi\in \mathcal{D}([-1,T)\times \TT)$ it holds
$$
\int_0^T\int_{\TT} u(-\pat \phi+\Lambda^\alpha\phi)+uB(u) \nabla \phi-f(u)\phi  \;dxds = \int_{\TT} u_0\phi(0) \;dx.
$$
\end{defi}

\section{Main results}\label{sec:3}

For clarity, let us introduce 
\[ [k]_+ = \begin{cases} k & \text{ for } 0 <k < \infty,\\
\text{an arbitrary finite number} & \text{ otherwise } (k <0 \text{ or } \infty)  \end{cases}
\]

Let us emphasize that when $k<0$ or $k=\infty$, $[k]_+$ can be chosen as large as required (see \eqref{equnif3} below).



\begin{teo}[Global-in-time weak solutions]\label{th:globalweak}
Let the active scalar relation in problem  \eqref{eqPKS} be given either by \eqref{f:B1} or by \eqref{f:B}.
Assume that $0<\alpha<2$ and $\chi,\beta,r>0$. Let $u_0\geq0$, $u_0\in L^{2}$ be an initial datum and let $T$ be any positive number. \\

(i) Case $2 r \geq \chi$: for every $0<\alpha<2$ there exists a global in time weak solution $u$  to \eqref{eqPKS}, in the sense of Definition \ref{def:weaksol}, such that
\[
\|u (t)\|_{ L^\infty (0,T;L^{1}) }\leq \max\{\|u_0\|_{L^1},4\pi^2\}.
\]
It enjoys the following extra regularity
$$
u\in L^2(0,T;H^{\alpha/2}) \cap L^\infty (0,T;L^{2}) \cap L^{2+2s} (0,T;W^{\left(\frac{\alpha}{2+2s}\right)^{-}, 1+s})  \cap L^m (0,T;L^{m})
$$
for any $0\le s \le 1$. There, $m=3$ for $2 r > \chi$ or any $m < 3$ for $2 r = \chi$. Moreover, if $u_0\in L^{p}$ for a $p \in (2, {[ \frac{\chi}{\chi - r}]_+}]$, then 
$$
u \in L^\infty (0,T;L^{p})  \cap L^{1+p^-} (0,T;L^{1+ p^-}) \quad \text{ and } \quad u^\frac{p}{2} \in L^2(0,T;H^{\alpha/2}).
$$ \\
(ii) Case $2r < \chi$: Assume that
\begin{equation}\label{eqhardr}
 \alpha >  4 \frac{ \chi ( \frac{\chi - 2r}{\chi - r}) - r}{{2\chi-r}}.
\end{equation}
There exists a global in time weak solution $u$  to \eqref{eqPKS}, in the sense of Definition \ref{def:weaksol}, such that
\[
\|u (t)\|_{ L^\infty (0,T;L^{1}) }\leq \max\{\|u_0\|_{L^1},4\pi^2\}.
\]
It enjoys the following extra regularity
$$
u \in L^\infty (0,T;L^{p_0})  \cap L^{1+p_0^-} (0,T;L^{1+ p_0^-}) \cap L^{2+2s} (0,T;W^{\left(\frac{\alpha}{2+2s}\right)^{-}, 1+s})  \quad \text{ and } \quad u^\frac{p_0}{2} \in L^2(0,T;H^{\alpha/2}).
$$
for any $0\le s \le  \min(1,  \frac{r}{\chi - r})$ and where $p_0 ={ \frac{\chi}{\chi - r}}$ ($<2$ in this case).
\end{teo}

\vskip 2mm
In a slightly narrower parameter range for the case $2r < \chi$, but still well into the supercritical regime $\alpha < 2$,  we obtain our central result
\begin{teo}[Global-in-time smooth solutions]\label{th:globalstrong}
Let the active scalar relation in problem \eqref{eqPKS} be given either by \eqref{f:B1} or by \eqref{f:B}.
Take any $0<\alpha<2$ and positive $\chi,\beta,r$. Let $u_0\geq0$, $u_0\in H^{k}$ with  $k > 1$, $k\in\mathbb{Z}^+,$ be an initial datum and let $T$ be any positive number. Assume that
\begin{equation}\label{eq:smooth}
\alpha>\max{ \left( 2 - \frac{2r}{\chi}, 0 \right) }.
\end{equation}
Then, the problem \eqref{eqPKS} admits a global in time classical solution,
$$
u\in C([0,T); H^k(\TT))
$$
that satisfies the bound
\begin{equation}\label{eq:infB}
\|u\|_{L^\infty(0,T;L^\infty(\TT))}\leq F(T,\|u_0\|_{L^{1}(\TT)},\|u_0\|_{L^\infty(\TT)},\alpha, r,\chi)
\end{equation}
for a nondecreasing $F$, finite for finite arguments.

Moreover, for $k > 3$, $k\in\mathbb{Z}^+,$ we have $ u \in C^{2,1} (\TT \times [0,T) )$ . If additionally $\alpha>1$, then $u$ is real analytic.
\end{teo}

\begin{remark} Theorem \ref{th:globalstrong} remains valid in the case of the spatial domain being $\mathbb{R}^2$. The proof is analogous, up to minor modifications.

\end{remark}

\section{Auxiliary results}
First we provide  for \eqref{eqPKS} a short-time smoothness result with a continuation criterion. Then, we recall the Stroock-Varopulous inequality and prove entropy estimates. Furthermore, needed energy estimates are given. Eventually, a nonlinear maximum principle for $\Lambda^\alpha$ is presented, which is needed for our proof of Theorem \ref{th:globalstrong}.
\subsection{Short-time smoothness and continuation criterion}

\begin{lem}\label{localexistence} Let $u_0\in H^4$ be a non-negative initial data. Then, if $0 < \alpha\leq 2$, there exists a time $0<T_{max}(u_0)\leq \infty$ such that there exists a non-negative solution 
$$
u \in C([0,T_{max}(u_0)],H^4) \quad \cap \quad C^{2,1} (\TT \times [0,T_{max}(u_0) ) ) 
$$
to the equation \eqref{eqPKS} with $B$ given by \eqref{f:B1} or by \eqref{f:B}. Moreover, if for a given $T$ the solution verifies the following bound
\begin{equation}\label{eq:contC}
\int_0^T \|u(s)\|_{L^\infty}ds<\infty,
\end{equation}
then it may be extended up to time $T+\delta$ for small enough $0<\delta$. Furthermore, under the restriction $1\leq \alpha $, the solution becomes real analytic. 
\end{lem}

\begin{proof}
The proof is similar to the one in \cite{AGM}. 
\end{proof}

\subsection{Entropy estimates}
\begin{lem}\label{lemaentropy2}
Let $0<s$,  \;$0<\alpha<2$, \; $0<\delta<\alpha/(2+2s)$ and $d \ge 1$. Then for a sufficiently smooth $u \ge 0$ ($u \in L^{\infty} (\Td) \cap H^\alpha (\Td)$ is enough) it holds
\begin{equation}\label{ene:1}
\frac{4s}{(1+s)^2}  \int_\Td  | \Lambda^\frac{\alpha}{2} (u^\frac{s+1}{2}) |^2dx \le \int_{\Td}\Lambda^\alpha u(x) u^s(x)dx.
\end{equation}
If additionally $s \le 1$, then
\begin{equation}\label{ene:2}
\|u\|_{\dot{W}^{\alpha/(2+2s)-\delta,1+s} (\Td)}^{2+2s}\leq C(\alpha,s,\delta, d)\|u\|_{L^{1+s} (\Td) }^{1+s} \int_{\Td}\Lambda^\alpha u(x) u^s(x)dx.
\end{equation}

Similarly, let $0<\alpha<2$, $0<\delta<\alpha/2$ and $d \ge 1$. Then for a sufficiently smooth $u \ge 0$ ($u \in H^\alpha (\Td)$ is enough) it holds
\begin{equation}\label{ene:3}
\|u\|_{\dot{W}^{\alpha/2-\delta,1}}^2\leq C(\alpha,d,\delta)\|u\|_{L^1}\int_{\Td}\Lambda^\alpha u(x)\log(u(x))dx.
\end{equation}
\end{lem}
\begin{proof}
Estimate \eqref{ene:1} is called sometimes the Stroock-Varopulos or the C\'ordoba-C\'ordoba inequality, compare \cite{cor2, liskevich11some}.  Inequalities \eqref{ene:2} and \eqref{ene:3} are multidimensional versions of, respectively, Lemma 7 and Lemma 6 in Appendix B of \cite{BG4}. For completeness, let us provide 
 a proof for \eqref{ene:2}.
Let us define 
$$
I=\int_\Td u^s(x)  \Lambda^\alpha u(x) dx.
$$
Using \eqref{Lkernelalpha} and changing variables, we compute
\begin{align*}
I&=c_{\alpha,d} \int_\Td \sum_{k\in \Zd, k\neq 0}\int_{\TT} u^s(x) \frac{u(x)-u(\eta)}{|x-\eta+2k\pi|^{d+\alpha}} d\eta dx +c_{\alpha,d} \int_\Td \text{P.V.}\int_{\Td} u^s(x)\frac{u(x)-u(\eta)}{|x-\eta|^{d+\alpha}} d\eta dx\\
&=c_{\alpha,d} \int_\Td \sum_{k\in \Zd, k\neq 0}\int_{\Td} u^s(\eta) \frac{u(\eta)-u(x)}{|\eta-x+2k\pi|^{d+\alpha}} d\eta dx +c_{\alpha} \int_\Td \text{P.V.}\int_{\Td} u^s(\eta)\frac{u(\eta)-u(x)}{|x-\eta|^{d+\alpha}}d\eta dx.
\end{align*}
In particular
\begin{equation}  \label{L7I1}
\begin{aligned}
I&\geq\frac{c_{\alpha,d}}{2}\int_\Td \int_{\Td} (u^s(x)-u^s(\eta))\frac{u(x)-u(\eta)}{|x-\eta|^{d+\alpha}} d\eta dx \\
&=  \frac{c_{\alpha,d}}{2} \int_\Td \int_{\Td}\int_0^1 \frac{d}{d \lambda}  \left( (\lambda u(x)+(1-\lambda)u(\eta))^s \right)     \frac{u(x)-u(\eta)}{|x-\eta|^{d+\alpha}} d\eta dx \\
& = \frac{c_{\alpha,d}}{2}  \int_\Td \int_{\Td}\int_0^{1} \frac{s}{(\lambda u(x)+(1-\lambda)u(\eta))^{1-s}}\frac{(u(x)-u(\eta))^2}{|x-\eta|^{d+\alpha}} d \lambda d\eta dx.
\end{aligned}
\end{equation}
Let us define 
$$
J=: \|u\|_{\dot{W}^{\beta,1+s}}^{1+s}=\int_\Td\int_\Td\frac{|u(x)-u(\eta)|^{1+s}}{|x-\eta|^{d+(1+s)\beta}}dxd\eta.
$$
Then we compute
\begin{align*}
J&=\int_\Td\int_\Td\int^1_0\frac{|u(x)-u(\eta)|^{1+s}}{|x-\eta|^{d+(1+s)\beta}}d\lambda dxd\eta\\
&=\int_\Td\int_\Td\int^1_0\frac{|u(x)-u(\eta)|}{|x-\eta|^{d/2+(1+s)\beta-\alpha/2}}\frac{|u(x)-u(\eta)|^s}{|x-\eta|^{d/2+\alpha/2}} \frac{|\lambda u(x)+(1-\lambda)u(\eta)|^{(1-s)/2}}{|\lambda u(x)+(1-\lambda)u(\eta)|^{(1-s)/2}} d\lambda dxd\eta\\
&=\int_\Td\int_\Td\int^1_0 F(x,\eta,\lambda)G(x,\eta,\lambda)d\lambda dx d\eta,
\end{align*}
where
$$
F=\frac{|u(x)-u(\eta)|}{|x-\eta|^{d/2+\alpha/2}}\frac{1}{|\lambda u(x)+(1-\lambda)u(\eta)|^{(1-s)/2}},
$$
$$
G=\frac{|u(x)-u(\eta)|^s}{|x-\eta|^{d/2+(1+s)\beta-\alpha/2}}|\lambda u(x)+(1-\lambda)u(\eta)|^{(1-s)/2}.
$$
Consequently 
\begin{equation}\label{L7I}
J\leq \|F\|_{L^2(\Td\times\Td\times[0,1])}\|G\|_{L^2(\Td\times\Td\times[0,1])}.
\end{equation}
On the one hand we have via \eqref{L7I1}
\[
\|F\|_{L^2(\Td\times\Td\times[0,1])}^2=\int_\Td\int_\Td \int_0^1\frac{(u(x)-u(\eta))^2}{|x-\eta|^{d+\alpha}}\frac{d\lambda dxd\eta}{(\lambda u(x)+(1-\lambda)u(\eta))^{1-s}} \leq \frac{2}{c_{\alpha} s} I.
\]
On the other  hand, since $s \le 1$
\begin{align*}
G^2&=\frac{|u(x)-u(\eta)|^{2s}}{|x-\eta|^{d/2+2(1+s)\beta-\alpha}}|\lambda u(x)+(1-\lambda)u(\eta)|^{1-s}\\
&\leq C(s)\frac{|u(x)|^{1+s}+|u(\eta)|^{1+s}}{|x-\eta|^{d+2(1+s)\beta-\alpha}} =  C(s)\frac{|u(x)|^{1+s}+|u(\eta)|^{1+s}}{|x-\eta|^{d- 2(1+s) \delta}},
\end{align*}
where for the equality we have made the choice $\beta=\frac{\alpha}{2+2s}-\delta$. Therefore, after integration and use of $\delta<\alpha/(2+2s) <1/(1+s)$, we arrive at
$$
\|G\|_{L^2(\Td\times\Td\times[0,1])}^2\leq C(s,\delta)\|u\|_{L^{1+s}}^{1+s}.
$$
Estimates for $F$ and $G$ plugged into \eqref{L7I} give \eqref{ene:2}. As mentioned at the beginning, inequality \eqref{ene:3} is an analogous multidimensionalization of  Lemma 6 in Appendix B of \cite{BG4}, as the just proved \eqref{ene:2} is in relation to Lemma 7 there.
\end{proof}

\subsection{A priori energy estimates} Let us define 
\begin{equation}\label{calN}
\mathcal{N}=\max\{\|u_0\|_{L^1},4\pi^2\}.
\end{equation}
\begin{lem}[Weak estimates]\label{Lemmanorms}  Let $0<T<\infty$ be arbitrary.  Let $u \in C ([0, T); H^2)$ be a non-negative solution to  \eqref{eqPKS} with $B$ given either by \eqref{f:B1} or by \eqref{f:B}. Take  any $0< s$ such that
\begin{equation}\label{lemnom1}
\frac{\chi s}{s+1} \le r.
\end{equation}
Then for any $t\in [0,T]$
\begin{align}
\sup_{t\in[0,T]}  \|u(t)\|_{L^1}&\leq \mathcal{N}, \label{aee1}\\
 \int_0^T\|u(\tau)\|^2_{L^2 }d \tau &\leq  \mathcal{N} (1/r+ T),  \label{aee2} \\
\sup_{t\in[0,T]} \|u(t) \|_{L^{s+1}} &\le e^{r T}\| u_0 \|_{L^{s+1}}, \label{aee3}  \\
 \int_0^T  \|u^\frac{s+1}{2} (\tau) \|^2_{H^\frac{\alpha}{2}} d \tau &\le \frac{{r(1+s)^2}}{4 s}  e^{r T}\| u_0^\frac{s+1}{2}  \|^2_{L^{2}} \label{aee3e}  \\
\left(  r (s+1) - \chi s \right)  \int_0^T \|u(\tau) \|^{s+2}_{L^{s+2}}  d \tau &\le  [r(s+1) e^{r(s+1)T } +1] \| u_0 \|^{s+1}_{L^{s+1}}, \label{aee4} 
\end{align}
If additionally $s \le 1$, then for an arbitrary $\delta\in \left(0,\frac{\alpha}{2+2s}\right)$ 
\begin{align}
\int_0^T\|u(\tau)\|_{\dot{W}^{\alpha/(2+2s)-\delta,1+s}}^{2+2s} d \tau &\leq C(\alpha,s,\delta) F_1(T,r,s)\| u_0 \|^{2s+2}_{L^{s+1}}, \label{aee5}
\end{align}
where
\[
F_1(T,r,s)=(rT+s+1) e^{2 r T}.
\]
Similarly,  for any $\delta\in \left(0,\frac{\alpha}{2}\right)$ 
\begin{align}
\int_0^T\|u(\tau)\|_{\dot{W}^{\alpha/2-\delta,1}}^2 d \tau&\leq C(\alpha, \delta)  \mathcal{N}F_2(u_0,T,r,\mathcal{N},\chi), \label{aee6}
\end{align}
where
$$
F_2(u_0,T,r,\mathcal{N},\chi)= \|u_0\|^2_{L^{2}} +  \mathcal{N} (\chi+r+ T  + 1).
$$
\end{lem}
\begin{proof}
In view of \eqref{eqPKS}, the ODE for $\|u(t)\|_{L^1}$ reads 
\begin{equation}\label{normL1}
\frac{d}{dt}\|u(t)\|_{L^1}=r\|u(t)\|_{L^1}-r\|u(t)\|_{L^2}^2.
\end{equation}
(We are allowed to write time derivatives of the space norms involved in this and the next proof, thanks to our qualitative assumption $u \in C ([0, T); H^2)$, which implies from equation \eqref{eqPKS} that $u_t \in L^\infty (L^2)$.)
Recalling Jensen's inequality
$
\|u(t)\|_{L^1}^2\leq 4\pi^2\|u(t)\|_{L^2}^2,
$
we get
$$
\frac{d}{dt}\|u(t)\|_{L^1}\leq r \|u(t)\|_{L^1}\left(1-\frac{1}{4\pi^2}\|u(t)\|_{L^1}\right).
$$
Hence \eqref{aee1}. 

Let us  integrate \eqref{normL1} between $0$ and a chosen $t$ to obtain
$$
\|u(t)\|_{L^1}-\|u_0\|_{L^1}= r \int_0^t\|u(s)\|_{L^1}ds-r\int_0^t\|u(s)\|_{L^2}^2ds,
$$
thus
$$
r \int_{0}^{t}\|u(s)\|_{L^2}^2\,ds\leq \|u_0\|_{L^1}-\|u(t)\|_{L^1}
+r \sup_{0\leq s\leq t}\|u(s)\|_{L^1}t\leq  \mathcal{N} + r \mathcal{N}t
$$
\emph{i.e.} \eqref{aee2}. 

Testing \eqref{eqDD} with $u^s$ we get after two integrations by parts
\begin{equation}\label{eq:ener}
\frac{1}{s+1} \frac{d}{dt} \int_\TT u^{s+1}dx +     \int_\TT u^s  \Lambda^\alpha u \, dx  = \int_\TT  \frac{\chi s}{s+1}  u^{s+1}  (\nabla \cdot B) (u)  +ru^{s+1} -ru^{s+2}dx.
\end{equation}
For the case of $B$ given by \eqref{f:B} we have straightforwardly 
\[
\nabla \cdot B (u) = u - \langle u \rangle \le u
\]

Our another choice \eqref{f:B1} for $B$ implies
\[
\nabla \cdot B (u)  = \Lambda^{\beta}(1+\Lambda^{\beta})^{-1}  (u) =:  \Lambda^{\beta} v,
\]
consequently, by the non-negativity of $u$, we have 
\[
u = v + \Lambda^{\beta} v \ge 0 \implies v  \ge -\Lambda^{\beta} v,
\]
hence by a nonlocal weak minimum principle 
\[
v  \ge \min v (x) = v (x_*) \ge -\Lambda^{\beta} v (x_*) \ge 0,
\]
therefore we obtain also in the case \eqref{f:B1} that
\begin{equation}\label{divB}
\nabla \cdot B (u) = u - v \le u.
\end{equation}


The inequality \eqref{eq:ener} gives via  \eqref{divB} and the Stroock-Varopulos inequality \eqref{ene:1}
\begin{equation}\label{eq:gr2}
 \frac{d}{dt} \int_\TT u^{s+1}  + \frac{4s}{1+s}  \int_\TT  | \Lambda^\frac{\alpha}{2} (u^\frac{s+1}{2}) |^2 dx + \left( r (s+1) - \chi s \right)  \int_\TT u^{s+2} \le r (s+1) \int_\TT  u^{s+1}. 
\end{equation}
Therefore, for $s$ satisfying \eqref{lemnom1} \emph{i.e.} $r (s+1) - \chi s \ge 0$, we get \eqref{aee3}.

Integrating \eqref{eq:gr2} in time and using next \eqref{aee3} we obtain \eqref{aee3e} and
\[
 \left(  r (s+1) - \chi s \right)  \int_0^t \int_\TT u^{s+2} (x, \tau) dx d \tau \le [r(s+1) e^{r(s+1)t } +1] \int_\TT u_0^{s+1} (x) dx,
 \]
which is \eqref{aee4}.

Using the inequality \eqref{ene:2} instead of \eqref{ene:1} for the second term of l.h.s. of  \eqref{eq:ener}, we arrive at

\begin{equation}\label{eq:ener:n1}
\|u (t) \|_{L^{1+s} }^{1+s}\frac{1}{s+1} \frac{d}{dt} \int_\TT u^{s+1}dx +  \frac{1}{C(\alpha,s,\delta)}  \|u (t) \|_{\dot{W}^{\alpha/(2+2s)-\delta,1+s}}^{2+2s}   \le r  \|u(t) \|^{2s+2}_{L^{s+1}}
\end{equation}
Time integration in \eqref{eq:ener:n1} together with \eqref{aee3} imply
\[
\int_0^T \|u (t) \|_{\dot{W}^{\alpha/(2+2s)-\delta,1+s}}^{2+2s} dt   \le  C(\alpha,s,\delta)\sup_{0\leq t \le T} \|u(t) \|^{2s+2}_{L^{s+1}} (rT + s+1),
\]
which gives \eqref{aee5}.

Let us define the functional
$$
\mathcal{F}=\int_\TT u\log(u)-u +1 \; dx.
$$

We have, integrating twice by parts
\[
\frac{d}{dt}\mathcal{F}=\int_\TT \pat u\log(u)dx \leq-\int_\TT \Lambda^\alpha u\log(u) dx+\chi\int_\TT u \,  (\nabla \cdot B) (u)  dx + r \int_\TT \log (u) dx.
\]
Consequently, integrating in time and applying \eqref{aee2}, we obtain
$$
\mathcal{F}(u(t))+\int_0^t\int_\TT \Lambda^\alpha u\log(u) dx\leq \mathcal{F}(u(0))+ \chi \mathcal{N} (1/r+ T) +r \mathcal{N} (1/r+ T + 4 \pi^2).
$$
This, together with 
$$
\mathcal{F}(u(0))\leq \int_{\{u(x,0)\geq1\}}u(x,0)^2 dx +4 \pi^2\leq \|u(0)\|^2_{L^2}+4\pi^2
$$
implies
\begin{equation}\label{uniformboundentropy}
\mathcal{F}(u(t))+\int_0^t\int_\TT \Lambda^\alpha (u)\log(u) \leq \|u_0\|^2_{L^{2}} + 2 \mathcal{N} (\chi /r+ T + 2 \pi^2 r + 1).
\end{equation}
Hence, using  \eqref{ene:3} and \eqref{aee1}, we get  \eqref{aee6}.
\end{proof}

\begin{lem}[Strong estimates]\label{Lemmanorms2}  Let $0<T<\infty$ be arbitrary.  Let $u \in C ([0, T); H^2)$ be a non-negative solution to  \eqref{eqPKS} with any $0<\alpha<2$, positive $\chi,\beta,r$ and $B$ given either by \eqref{f:B1} or by \eqref{f:B}. If
\begin{equation}\label{apstrongC}
\alpha > 2\left(1-\frac{r}{\chi}\right)
\end{equation}
then for any finite $p \ge 1$ there exist a finite $C_1  ({p,\chi, r, \mathcal{N}})$ and $C_2  ({p,\chi, r, \alpha})$ such that
\begin{equation}\label{apstrong}
\|  u \|_{L^\infty (0,T; L^{p})} \le  C_1  \left( e^{C_2 T} +1\right)  \| u_0 \|_{L^{p}}^{C_2}.
\end{equation}
\end{lem}
\begin{proof}
Let us recall \eqref{eq:gr2}, valid a priori for any $0<s<\infty$. We have not used the dissipation in the classical way there yet. Putting the last summand of l.h.s. of   \eqref{eq:gr2} to its r.h.s. and using the Sobolev embedding 
$$
H^{\alpha/2}\subset L^{\frac{2}{1-\alpha/2}}
$$
for $u^\frac{s+1}{2}$, we obtain
\begin{equation}\label{eq:enerd}
 \frac{d}{dt} \int_\TT u^{s+1}dx +   c_{s,\alpha} \left( \int_\TT  | u|^\frac{s+1}{1-\alpha/2} dx \right)^{1- \alpha/2}  \le \int_\TT  {\chi s}  u^{s+2}  +r (s+1) u^{s+1} -r (s+1) u^{s+2} dx.
\end{equation}
Let $m>0$ be a number, to be precised further. Inequality $u^{s+1} - u^{s+2} \le 1$ and
interpolation
\[
 \|  u \|_{L^{s+2}}   \le C \| u \|^{\theta}_{L^\frac{s+1}{1-\alpha/2}}   \| u \|^{1- \theta}_{L^m }
\]
with $\theta = \frac{\frac{m}{s+2} -1}{\frac{m (2-\alpha)}{2(s+1)}-1}$
 imply
\begin{equation*}
 \frac{d}{dt} \int_\TT u^{s+1}dx +   c_{s,\alpha}  \|  u \|^{s+1}_{L^\frac{s+1}{1-\alpha/2}}   \le {\chi s}  \| u \|^{\theta (s+2)}_{L^\frac{s+1}{1-\alpha/2}}   \| u \|^{(1- \theta) (s+2)}_{L^m } + 4 \pi^2 r (s+1) 
\end{equation*}
as long as the interpolation holds, \emph{i.e.}
\begin{equation}\label{eq:enerdc}
\alpha (s+2) > 2 \quad \text{ and } \quad m < s+2.
\end{equation}
Hence the Young inequality yields
\[
 \frac{d}{dt} \int_\TT u^{s+1}dx +   \frac{c_{s,\alpha}}{2}  \|  u \|^{s+1}_{L^\frac{s+1}{1-\alpha/2}}   \le C ({s,\chi})  \| u \|^{ \frac{(1- \theta) (s+2) (s+1)}{(s+1) - \theta (s+2)}}_{L^m} + 4 \pi^2 r (s+1).
\]
Recalling the form of $\theta$, we arrive at
\begin{equation}\label{eq:enerd2}
 \frac{d}{dt} \int_\TT u^{s+1}dx +   \frac{c_{s,\alpha}}{2} \|  u \|^{s+1}_{L^\frac{s+1}{1-\alpha/2}}   \le C ({s,\chi, r})  \left(  \| u \|^{\frac{m (\alpha (s+2) -2)}{m \alpha - 2}}_{L^m} + 1 \right).
\end{equation}
Integrating \eqref{eq:enerd2} in time and neglecting the second term on its l.h.s., one has
\begin{equation}\label{eq:enerd21}
\|  u \|_{L^\infty (0,T; L^{s+1})} \le  C  ({s,\chi, r})  \left( | u |^{\frac{\gamma}{ s+1}}_{L^\gamma (0,T;L^{m})} + T^\frac{1}{{s+1}} \right) + \|  u_0 \|_{ L^{s+1}},
\end{equation}
where
\begin{equation}\label{eq:enerd2gamma}
\gamma =: \frac{m (\alpha (s+2) -2)}{(m \alpha - 2)}.
\end{equation}
Let us estimate coarsely \eqref{eq:enerd21} and obtain 
\begin{equation}\label{eq:enerd22}
\|  u \|_{L^\infty (0,T; L^{s+1})} \le  C  ({s,\chi, r})  \left( \| u \|^{\frac{\gamma}{ s+1}}_{L^\infty (0,T;L^{m})} + T^\frac{1}{{s+1}} \right) + \|  u_0 \|_{ L^{s+1}}.
\end{equation}
It is valid as long as \eqref{eq:enerdc} and $\gamma < \infty$ hold, \emph{i.e.} 
\begin{equation}\label{eq:enerd3a}
\alpha (s+2) > 2, \quad m < s+2, \quad m \alpha > 2.
\end{equation}

Observe that if in \eqref{eq:enerd22} one had chosen $m=1$ given by \eqref{aee1}, the needed finiteness of \eqref{eq:enerd2gamma} would coerce $\alpha>2$, compare the last inequality of \eqref{eq:enerd3a}. Fortunately, the logistic term provided us with \eqref{aee3}. For $r \ge \chi$ it implies our thesis immediately. For $r < \chi$, choosing the largest integrability exponent allowed in \eqref{aee3}, \emph{i.e.} $\frac{r}{\chi - r}$, we have 
\[
\sup_{t\in[0,T]}\|u(t) \|_{L^{\frac{\chi}{\chi - r}}} \le e^{r T}\| u_0 \|_{L^{\frac{\chi}{\chi - r}}}.
\]
Consequently, we can let $m = \frac{\chi}{\chi - r}$ in \eqref{eq:enerd3a} and obtain

\begin{equation}\label{eq:enerd4}
\|  u \|_{L^\infty (0,T; L^{s+1})} \le  C  ({s,\chi, r})  \left( (e^{r T}\| u_0 \|_{L^{\frac{\chi}{\chi - r}}})^{\frac{\gamma}{ s+1}}+ T^\frac{1}{{s+1}} \right) + \|  u_0 \|_{ L^{s+1}},
\end{equation}
for which suffices, according to  \eqref{eq:enerd3a},
\begin{equation}\label{eq:enerd3b}
s+1 > \frac{2}{\alpha} -1,  \quad s+1 > \frac{\chi}{\chi - r} , \quad  \alpha >  \frac{2}{m}= 2\left(1-\frac{r}{\chi}\right).
\end{equation}
In view of the last condition of  \eqref{eq:enerd3b}, for the first condition there suffices $s+1 >  \frac{r}{\chi - r}$, which is in our case $r < \chi$ weaker than the middle one. Interpolation in  \eqref{eq:enerd4} of $\| u_0 \|_{L^{\frac{\chi}{\chi - r}}}$ between $\|  u_0 \|_{ L^{1}}$ and $\|  u_0 \|_{ L^{s+1}}$ yields \eqref{apstrong}.
\end{proof}

\begin{remark}
Estimate \eqref{eq:enerd21} or \eqref{eq:enerd22} may look appealing as a possible building block for a inductive Moser-type procedure, aimed at obtaining bound for $\|  u \|_{L^\infty (0,T; L^{\infty})} < \infty$, which would immediately imply our main Theorem \ref{th:globalstrong}. But the power ${\frac{\gamma}{ s+1}}$ in \eqref{eq:enerd21} exceeds $1$.
\end{remark}
\begin{remark}
Instead of using \eqref{eq:enerd22}, one can try to improve a little the admissible range of $\alpha$'s by using \eqref{eq:enerd21} bounded by an interpolation of \eqref{aee3} and \eqref{aee4}.
\end{remark}

\subsection{A nonlinear maximum principle}\label{sec:3b}
In this section we prove a new nonlinear maximum principle in the same spirit as the ones in \cite{constantin2012nonlinear, cor2}. 

\begin{teo}\label{th:max} Let $\delta \in (0,1)$, $\partial_{x_1}\partial_{x_2}\phi=u\in C^2(\TT)$ and denote by $x^*$ the point such that
$$
\max_{x\in \TT} u(x)=u(x^*).
$$
Then we have the following alternative:
\begin{enumerate}
\item Either 
$$
\|u\|_{L^\infty}\leq C_{\alpha, \delta} ^1\|\phi\|_{C^\delta},
$$

\item or
$$
\Lambda^\alpha u(x^*)\geq C^2_{\alpha, \delta} \frac{u(x^*)^{1+\frac{\alpha}{2-\delta}}}{\|\phi\|_{C^\delta}^{1+\frac{\alpha}{2-\delta}}},
$$
\end{enumerate}
where $C^i_{\alpha, \delta} $ are explicit constants depending only on $\alpha$ {and $\delta$}.
\end{teo}
\begin{proof}
Define a smooth cutoff function $\zeta(x)$ such that $\zeta(x)=0$ if $|x|\leq 1$, $\zeta(x)=1$ if $|x|\geq 2$. Notice that $\zeta$ can be taken such that $\|\zeta\|_{C^2}\leq \mathcal{C}$ for a large enough $\mathcal{C}$. Let $\mathfrak{R}>0$ be a constant such that
\begin{equation}\label{required}
\mathfrak{R}<\frac{\pi}{2^{1+\alpha}}.
\end{equation}
This constant will be fixed later.

Using the kernel representation for $\Lambda^\alpha$ and integrating by parts, we have
\begin{align*}
\frac{\Lambda^\alpha u(x^*)}{c_{\alpha}}&\geq \text{P.V.}\int_{{\TT}} \frac{u(x^*)-u(x^*-y)}{|y|^{2+\alpha}}dy\\
&\geq \text{P.V.}\int_{{\TT}} \frac{u(x^*)-\partial_{y_1}\partial_{y_2}(\phi(x^*)-\phi(x^*-y))\zeta(y/\mathfrak{R})}{|y|^{2+\alpha}}dy\\
&\geq u(x^*)\int_{{\TT}\cap B^c_{2\mathfrak{R}}} \frac{1}{|y|^{2+\alpha}}dy-\|\phi\|_{C^\delta}\int_{{\TT}} |y|^{\delta}\left|\partial_{y_1}\partial_{y_2}\left(\frac{\zeta(y/\mathfrak{R})}{|y|^{2+\alpha}}\right)\right|dy\\
&\geq c_1\frac{u(x^*)}{\mathfrak{R}^\alpha}-c_2\frac{\|\phi\|_{C^\delta}}{\mathfrak{R}^{2+\alpha-\delta}},
\end{align*}
where in the last equation we have used \eqref{required}. Now let us choose
$$
\mathfrak{R}=c_3\frac{\|\phi\|_{C^\delta}^{1/(2-\delta)}}{u(x^*)^{1/(2-\delta)}},
$$
and obtain that, if $\mathfrak{R}$ fulfils \eqref{required}, 
$$
\Lambda^\alpha u(x^*)\geq C^2_\alpha\frac{u(x^*)^{1+\frac{\alpha}{2-\delta}}}{\|\phi\|^{\frac{\alpha}{2-\delta}}}.
$$
Otherwise, $\mathfrak{R}$ does not fulfil \eqref{required} and we conclude that
$$
\|u\|_{L^\infty}\leq C^1_\alpha\|\phi\|_{C^\delta}. 
$$
\end{proof}

\begin{remark}
A similar result holds in the case $u\in \mathcal{S}(\RR^2)$.
\end{remark}

\section{Proof of Theorem \ref{th:globalweak}}\label{sec:5}

\subsection{Approximate problems}
Let $T>0$ be an arbitrary fixed number. For $\epsilon \in (0,1)$, we define the following {approximate} problem on $\TT$
\begin{equation}\label{eqapprox}
\pat u_\epsilon=\epsilon\Delta u_\epsilon-\Lambda ^\alpha u_\epsilon+ \chi \pax(u_\epsilon B(u_\epsilon))+ ru (1-u),
\end{equation}
where $B$ agrees with \eqref{f:B1} or with \eqref{f:B}. The mollified initial datum $u_\epsilon$ is given as
$$
u_\epsilon(0,x)=\jeps u_0(x)\geq0.
$$
Along the lines of Lemma \ref{localexistence} we have local existence of a smooth solution $u_\epsilon$. Furthermore, recalling Theorem 2.5 of  \cite{TelloWinkler}, we have global existence of regular solutions. Namely, the differences between the system considered in  \cite{TelloWinkler} and  \eqref{eqapprox} with $B$ given by \eqref{f:B1} are: presence of a smoothening $-\Lambda ^\alpha$ in \eqref{eqapprox} and more basic `boundary conditions' in \eqref{eqapprox}. In case of $B$ given by \eqref{f:B} the proof is analogous.

\subsection{Uniform estimates}
We will use in what follows the thesis of Lemma \ref{Lemmanorms} for $u_\epsilon$ in place of $u$, since it holds for  $u_\epsilon$ by the same token as for $u$. Recall that we denote
\[ [k]_+ = \begin{cases} k & \text{ for } 0 <k < \infty,\\
\text{an arbitrary finite number} & \text{ otherwise } (k <0 \text{ or } \infty).  \end{cases}
\]
Lemma \ref{Lemmanorms}  and the Young inequality for convolutions give the following $\epsilon$-independent bounds
\begin{equation}\label{equnif1}
\sup_{0\leq t\leq T}\|u_\epsilon(t)\|_{L^{1}}\leq \mathcal{N},
\end{equation}

\begin{equation}\label{equnif4}
\int_0^T\|u_\epsilon(\tau)\|^{2}_{L^{2}}d\tau \leq   \mathcal{N} (1/r+ T).
\end{equation}
Furthermore, the largest \[s_0 = \left[ \frac{r}{\chi - r}\right]_+\] following from equality in \eqref{lemnom1}, yields for any $1 < p \le {[ \frac{\chi}{\chi - r}]_+}$

\begin{equation}\label{equnif3}
\sup_{0\leq t\leq T}\|u_\epsilon(t)\|_{L^{p}}\leq e^{r T}\| u_0 \|_{L^{p}},
\end{equation}
\begin{equation}\label{equnif4a}
\int_0^T  \|u_\epsilon^\frac{p}{2} (\tau) \|^2_{H^\frac{\alpha}{2}} d \tau \le C(r, s)\;  e^{r T}\| u_0^\frac{p}{2} \|^2_{L^{2}},
 \end{equation}
 where  in the case $r \geq \chi/2$ we can have $p=2$,
as well as
\begin{equation}\label{equnif41}
\int_0^T \|u_\epsilon (\tau) \|^{1+p^-}_{L^{1+ p^-}}  d \tau \le C (r, p^-, \chi) [ e^{(r p^-)T } +1] \| u_0 \|^{p^-}_{L^{p^-}}.
\end{equation}
We have also
\begin{equation}\label{equnif2}
 \int_0^T\|u_\epsilon(\tau)\|_{\dot{W}^{\alpha/2-\delta,1}}^2 d \tau \leq C(\alpha, \delta)  \mathcal{N}F_2(u_0,T,r,\mathcal{N},\chi)  \qquad \text{ for any }   0<\delta<\alpha/2,
\end{equation}
where
$$
F_2(u_0,T,r,\mathcal{N},\chi)= \|u_0\|^2_{L^{2}} +  \mathcal{N} (\chi+r+ T  + 1).
$$
Finally, for any $s \le \min(1, s_0)$ it holds also
\begin{equation}\label{equnif42}
\int_0^T\|u_\epsilon (\tau)\|_{\dot{W}^{\alpha/(2+2s)-\delta,1+s}}^{2+2s} d \tau \leq C(\alpha,s,\delta) F_1(T,r,s)\| u_0 \|^{2+2s}_{L^{1+s}} \qquad  \text{ for any } 0<\delta<\alpha/(2+2s),
\end{equation}
where
\[
F_1(T,r,s)=(rT+s+1) e^{2 r T}.
\]

\subsection{Convergence}
Due to the uniform bound \eqref{equnif4} we have 
\begin{equation}\label{weak_tder}
\pat u_\epsilon\in L^2(0,T;H^{-2})
\end{equation}
$\epsilon$-uniformly bounded via a generous estimate of $\pat u_\epsilon$ by the remainder of  \eqref{eqapprox}.

Let us recall the Aubin-Lions Theorem. If both $q_1, q_0 \in (1, \infty)$, then for a Gelfand's triple of reflexive Banach spaces $X_0 \subset \subset X \subset X_1$, the evolutionary space
$$
Y=\{h: \, h\in L^{q_0} (0,T;X_0), \; \pat h \in L^{q_1} (0,T;X_1)\}
$$
is compactly embedded in $L^{q_0} (0,T;X)$. 

\subsubsection{Case $2 r\geq \chi$} Let us first consider the case $2 r\geq \chi$, where \eqref{equnif4a} holds with $p =2$. Hence we can extract a subsequence (hare and later, we denoted a subsequence again by $u_\epsilon$) such that
$$
u_\epsilon\rightharpoonup u \text{ in }L^2(0,T;H^{\alpha/2}).
$$
Let us take the following Gelfand's triple
$$
X_{1}=H^{-2},\,X=L^2,\,X_0=H^{\alpha/2}.
$$
The Aubin-Lions Theorem with $q_0=q_1=2$  implies via \eqref{equnif4a} that
$$
u_\epsilon\rightarrow u \text{ in }L^2(0,T;L^2).
$$
This strong compactness is enough to pass to the limit in the weak formulation, since $B$ of the nonlinear term $u B(u)$ is a compact operator in $L^2$. As a consequence, $u$ is a global weak solution (in the sense of Definition \ref{def:weaksol}) of the original system \eqref{eqDD}. Its regularity follows from the uniform bounds \eqref{equnif1} -- \eqref{equnif42} and the lower weak (or $*$-weak,  where applicable) continuity of respective norms.

\subsubsection{Case $2 r < \chi$}  Let us explain how to adapt the proof for the second case  $r < \chi/2$, where  the largest $s_0 = [ \frac{r}{\chi - r}]_+ =\frac{r}{\chi - r}  \in (0,1)$ and   $p \le {[ \frac{\chi}{\chi - r}]_+} =  \frac{\chi}{\chi - r} < 2$, in view of the bound on $p$ of \eqref{equnif3} --  \eqref{equnif41}. Instead of  \eqref{equnif4a}, we intend to use in this case \eqref{equnif42}. Notice that for $s \le s_0 =\frac{r}{\chi - r}  \in (0,1)$
one has 
\[W^{(\alpha/(2+2s))^-,1+s}\subset\subset L^\xi, \quad \xi < \frac{4+4s}{4-\alpha}.\]
We choose now the related Gelfand's triple
$$
X_{1}=H^{-2},\,X=L^\xi,\,X_0=W^{(\alpha/(2+2s))^-,1+s}.
$$
Hence the Aubin-Lions Theorem with $q_1 = 2, q_0 = 2+2s$ implies that one can extract a subsequence (denoted again by $u_\epsilon$) such that
$$
u_\epsilon\rightarrow u \text{ in }L^{2+2s}(0,T;L^\xi).
$$
Let us observe that a limit passage in the chemotactic term $u_\epsilon B (u_\epsilon)$ is easier than in the term $u^2_\epsilon$, since $B$ is a compact operator in $L^2$. Let us therefore 
consider the limit passage in the most troublesome part $-r u^2_\epsilon $
\begin{equation}\label{eq:relaxC}
r \int_0^T \int_\TT  | (u^2_\epsilon - u^2) \phi| \le r \|\phi\|_{L^\infty (L^\infty)}  \| u_\epsilon - u\|_{L^{2+2s} (L^\xi)}  \| u_\epsilon +u\|_{L^{\frac{2+2s}{1+2s}} (L^{\xi'})}.
\end{equation}
In order to control the last term of \eqref{eq:relaxC}, we want to use additionally \eqref{equnif41}. To this end one needs $\xi' < 1+ p^-$ (the other condition related to the time integrability always holds) or equivalently $\xi' < \frac{\chi}{\chi - r} +1$, \emph{i.e.} 
\[
\xi' < \frac{2\chi-r}{\chi - r} \iff \frac{4+4s}{\alpha +4s} < \frac{2\chi-r}{\chi - r}  \iff 4 \frac{ \chi (1-s) - r}{{2\chi-r}} < \alpha
\]
\emph{i.e.}  \eqref{eqhardr} for the best possible $s=s_0$.
Consequently, the right-hand-side of \eqref{eq:relaxC} vanishes as $\epsilon \to 0$ thanks to strong convergence in $L^{2+2s}(0,T;L^\xi)$ and boundedness in $L^{1+p^-}( 0, T; L^{1+p^-}) \subset {L^{\frac{2+2s}{1+2s}} (L^{\xi'})} $.

\section{Proof of Theorem \ref{th:globalstrong}}\label{sec:6}
This proof is independent from the proof of Theorem \ref{th:globalweak}. For the sake of clarity in the exposition, let us consider the case $u_0 \in H^k$ with $k=4$, the cases $k>4$ being similar. Eventually, we will recover the cases $k=2,3$ (compare Section \ref{ssec:low}).

Let us take any $u_0 \in H^4$ and use Lemma \ref{localexistence} to obtain a classical local-in-time solution on the time interval $[0, T_{max})$:
$$
u\in C([0,T_{max}); H^4(\TT))   \quad \cap \quad C^{2,1} (\TT \times [0,T_{max}(u_0) ) )
$$ 
to \eqref{eqPKS}. The standard continuation argument for autonomous ODE in Banach spaces implies that either $T_{max}=\infty$ or $T_{max}<\infty$ and 
$$
\limsup_{t\rightarrow T_{max}}\|u(t)\|_{H^4(\TT)}=\infty.
$$
Then, due to Lemma \ref{localexistence}, an estimate showing 
\begin{equation}\label{desired}
\|u\|_{L^\infty(0,T_{max};L^\infty(\TT))}\leq G(T_{max})<\infty,
\end{equation}
provides us with the estimate
$$
\limsup_{t\rightarrow T_{max}}\|u(t)\|_{H^4(\TT)}\leq F(T_{max})<\infty,
$$
and thus, it proves the global existence of solution.

As in Theorem \ref{th:max}, let us denote by $x_t^*$ the point such that
$$
u(x_t^*,t)=\max_{x\in \TT} u(x,t).
$$
Let us also define 
$$
\bar{u}(t)=u(x_t^*,t).
$$

Since $u\in C^1(\TT\times (0,T_{max}))$, the function $\bar{u}(t)$ is Lipschitz and, consequently, it has a derivative almost everywhere. Using vanishing of a derivative at the point of maximum, we see that
$$
\frac{d}{dt}\bar{u}(t)=\partial_t u(x^*_t,t).
$$
Inequality \eqref{divB} for $B$ given by  \eqref{f:B1} implies that the evolution of $\bar u (t)$ follows
\begin{equation}\label{ODE}
\frac{d}{dt} \bar u +  \Lambda^\alpha  u(x^*_t,t) \leq \chi\bar{u}^2+ r\bar{u}(1-\bar{u}).
\end{equation}

\subsection{The most-damped case $r\geq\chi$}
Due to the weak maximum principle we obtain from \eqref{ODE}
$$
\frac{d}{dt} \bar u\leq \chi\bar{u}^2+ r\bar{u}(1-\bar{u}),
$$
and, if $r\geq\chi$, we conclude
$$
\bar{u}(t)\leq e^{rt}\bar{u}(0).
$$

\subsection{The least-damped case $2r < \chi$}
First, let us define 
$$
\phi=\int_{-\pi}^{x_1}\int_{-\pi}^{x_2}u(y_1,y_2)dy_{1}dy_2
$$
and recall that the best possible choice of exponent:
\begin{equation}\label{eq:sch}
p_0= \frac{\chi}{\chi -r}
\end{equation}
in  \eqref{aee3}  implies
\begin{equation}\label{aaar}
\sup_{0\leq t\leq T}\|u (t)\|_{L^{p_0}}\leq e^{r T}\| u_0 \|_{L^{p_0}},
\end{equation}
compare \eqref{equnif3}. Notice that $p_0 \in(1, 2)$. It holds
$$
|\phi(x+h)-\phi(x)|= \left|\int_{x_1}^{x_1+h_1}\int_{x_2}^{x_2+h_2}u(y_1,y_2)dy_{1}dy_2\right|\leq \|u\|_{L^{p_0}}(h_1h_2)^{\frac{p_0-1}{p_0}} \le  \|u\|_{L^{p_0}} 2^{\frac{1-p_0}{p_0}} |h|^{\frac{2(p_0-1)}{p_0}}
$$
thus
\begin{equation}\label{max:p1}
\|\phi\|_{C^{\frac{2(p_0-1)}{p_0}}}\leq c\|u\|_{L^{p_0}}.
\end{equation}
Inequality \eqref{max:p1} used in our nonlinear maximum principle (Theorem \ref{th:max}) with $\delta={\frac{2(p_0-1)}{p_0}} \; (<1)$ implies the following alternative:

\begin{align}
&\text{Either } \quad  \|u\|_{L^\infty}\leq C_\alpha^1\|\phi\|_{C^\delta} \leq c\,   C_\alpha^1 \|u\|_{L^{p_0}} \label{max:l1} \\
&\text{or } \quad  \Lambda^\alpha u(x^*)\geq C^2_\alpha\frac{u(x^*)^{1+\frac{\alpha}{2-\delta}}}{\|\phi\|_{C^\delta}^{1+\frac{\alpha}{2-\delta}}} \geq C^2_\alpha\frac{u(x^*)^{1+\frac{\alpha p_0}{2}}}{\|\phi\|_{C^{\frac{2(p_0-1)}{p_0}} }^{1+\frac{\alpha p_0}{2}}} \geq C^2_\alpha \frac{u(x^*)^{1+\frac{\alpha p_0}{2}}}{c\|u\|^{1+\frac{\alpha p_0}{2}}_{L^{p_0}}}. \label{max:l2}
\end{align}
Consequently, inequalities \eqref{aaar}, \eqref{max:l1} and \eqref{max:l2} yield for $u (t) $ being a classical local in time solution to \eqref{eqPKS}
\begin{align}
&\text{Either } \quad   \|u \|_{L^\infty (0,T_{max};L^\infty(\TT))}   \le c \;C_\alpha^1 \| u \|_{L^\infty (0,T_{max};L^{p_0}(\TT))} \le c \;C_\alpha^1 e^{r T}\| u_0 \|_{L^{p_0}} \le  C_1,  \label{max:l1s} \\
&\text{or } \quad  \Lambda^\alpha u(x_t^*,t) \geq  \frac{C^2_\alpha      \bar u(t)^{1+\frac{\alpha p_0}{2}}}{c \| u \|^{1+\frac{\alpha p_0}{2}} _{L^\infty (0,T_{max};L^{p_0}(\TT))}} \ge   \frac{C^2_\alpha      \bar u(t)^{1+\frac{\alpha p_0}{2}}}{c  e^{{(1+\frac{\alpha p_0}{2})} r T}\| u_0 \|^{1+\frac{\alpha p_0}{2}}_{L^{p_0}} }
 \ge C^{-1}_1 \bar u(t)^{1+\frac{\alpha p_0}{2}}, \label{max:l2s}
\end{align}
where $C_1 = C_1 (\alpha, r, T, p_0, \| u_0 \|_{L^{p_0}}) $, but, more importantly, $C_1$ being $T_{max}$-independent and finite for any $T$ (for nonzero data). Inequality \eqref{max:l1s} has already the desired form \eqref{desired}. Our aim is now to show that \eqref{max:l2s} also implies \eqref{desired}. Equation \eqref{ODE} together with \eqref{max:l2s} gives
\begin{equation}\label{odeh}
\frac{d}{dt}\bar u + C^{-1}_1   \bar u^{1+\frac{\alpha p_0}{2}}   \le (\chi-r) \bar u^2 + r\bar{u},
\end{equation}
hence, as long as it holds  ${1+\frac{\alpha p_0}{2}} >2$ or, equivalently,  \eqref{eq:smooth} (via \eqref{eq:sch}), we obtain
\[
\frac{d}{dt} \bar u   \le C_2 
\]
Integration in time gives us again \eqref{desired}.

\begin{remark}
It may be appealing to use a stronger estimate of Lemma \ref{Lemmanorms2} instead of \eqref{aaar}, in order to obtain a wider admissibility range for $\alpha$'a than  \eqref{eq:smooth}. However, let us observe that the assumption \eqref{apstrongC} of Lemma \ref{Lemmanorms2} coerces immediately \eqref{eq:smooth}. (interestingly, we obtained the same bound for our pointwise proof and for the strong estimates.) It may  suggest that Lemma \ref{Lemmanorms2} is useless, since it provides merely $L^\infty (L^p)$ estimates precisely in the range where we just proved smoothness (classical solutions). It turns out that Lemma \ref{Lemmanorms2} will turn essential as a patch in the following section.
\end{remark}

\subsection{The intermediate case $r<\chi\leq 2r$}
Recovering the full admissible range  \eqref{eq:smooth}, \emph{i.e.} $\alpha>\max{ \left( 2 - \frac{2r}{\chi}, 0 \right) }$ turns out to bear an unexpected difficulty in this case. Namely, now \eqref{eq:sch} in \eqref{aaar} implies the optimal $p_0 =\frac{\chi}{\chi -r} \ge 2$, consequently the related  $\delta={\frac{2(p_0-1)}{p_0}} \ge 1$, whereas our nonlinear maximum principle (Theorem \ref{th:max}) allows for $\delta \in (0,1)$. Of course we can choose a suboptimal $p_0 =2^-$ and the related $\delta= 1^-$, but then to close the estimate \eqref{odeh}, we need $\alpha>1$, which is an additional bound in this intermediate range  $r<\chi<2r$ compared to  \eqref{eq:smooth}. In order to cope with this obstacle, let us use another nonlinear maximum principle  in place of Theorem \ref{th:max}, namely Lemma 1 of \cite{Gsemiconductor}, Appendix A. It implies in place of the dichotomy  \eqref{max:l1} -- \eqref{max:l2}
\begin{equation}
\Lambda^\alpha u(x_t^*)  \geq C (\alpha,p) \frac{u(x^*)^{1+\frac{\alpha p}{2}}}{\|u\|^{1+\frac{\alpha p}{2}}_{L^{p}}}. \label{max:l2_n}
\end{equation}
and, consequently, in place of  \eqref{max:l1s} -- \eqref{max:l2s} (we choose immediately $p_0=\frac{\chi}{\chi -r}$)
\begin{equation}
 \Lambda^\alpha u(x_t^*,t) \geq  \frac{C \left(\alpha, \frac{\chi}{\chi -r}\right)      \bar u(t)^{1+\frac{\alpha \frac{\chi}{\chi -r}}{2}}}{\| u \|^{1+\frac{\alpha \frac{\chi}{\chi -r}}{2}} _{L^\infty (0,T_{max};L^{\frac{\chi}{\chi -r}}(\TT))}} \ge   \frac{C \left(\alpha, \frac{\chi}{\chi -r}\right)    \bar u(t)^{1+\frac{\alpha \frac{\chi}{\chi -r}}{2}}}{ \left(C_1  \left( e^{C_2 T} +1\right)  \| u_0 \|_{L^{\frac{\chi}{\chi -r}}}^{C_2} \right)^{1+\frac{\alpha \frac{\chi}{\chi -r}}{2}}}
 \ge \delta_1 \bar u(t)^{1+\frac{\alpha \frac{\chi}{\chi -r}}{2}}, \label{max:l2s_n}
\end{equation}
where for the second inequality we used \eqref{apstrong} of Lemma \ref{Lemmanorms2}. Here \begin{equation}\label{eq:delta}
\delta_1 = \delta_1 \left(\alpha, r, \chi, T, \frac{\chi}{\chi -r}, \| u_0 \|_{L^{\frac{\chi}{\chi -r}}}\right) >0.
\end{equation}
Let us assume that the solution has a finite maximal lifespan $T_{max}$. In particular, due to Lemma \ref{localexistence} and condition \eqref{eq:contC}, we have that
\begin{equation}\label{eq:contC2}
\int_0^{T_{max}}\|u(s)\|_{L^\infty}ds=\infty.
\end{equation}
Using \eqref{ODE}, we have the inequality
\begin{equation}\label{ODE2}
\frac{d}{dt} \bar u + \delta_1 \bar{u}^{\gamma} \leq \chi\bar{u}^2+ r\bar{u}(1-\bar{u}),
\end{equation}
where 
$$
\gamma:=1+\frac{\alpha \frac{\chi}{\chi -r}}{2}>2,
$$
because in the range $\alpha>\max{ \left( 2 - \frac{2r}{\chi}, 0 \right) }$ it holds
$$
1+\frac{\alpha \frac{\chi}{\chi -r}}{2}>1+\max{ \left( \frac{\chi-r}{\chi}, 0 \right) } \frac{\chi}{\chi -r}=2.
$$
Inequality \eqref{ODE2} implies the bound
$$
\sup_{0\leq t \leq T_{max}}\|u(t)\|_{L^\infty}\leq C(\alpha,r,u_0,\chi,T_{max}).
$$
This is a contradiction with \eqref{eq:contC2}.

The proof of the statement of Theorem \ref{th:globalstrong} for $k \ge 4$, including the last sentence of Theorem \ref{th:globalstrong}, is complete.

\begin{remark}
Let us note that the recent argument for case $r<\chi \le 2r$, based on Lemma \ref{Lemmanorms2} and Lemma 1 of \cite{Gsemiconductor}, works also fine for all the other cases. However, we decided in the case $2r< \chi$ for a proof that uses (and thus presents in action) a new nonlinear maximum principle, since we believe it may be of independent interest.
\end{remark}
\subsection{Lowering the initial regularity}\label{ssec:low}
Finally, in order to allow for a lower regularity data $u_0 \in H^k$, $k>1$ observe that the entire proof can be made for a solution to \eqref{eqPKS} with a mollified $u^\epsilon_0 \in H^4$. We obtain then a classical solution $u^\epsilon$ with an $\epsilon-$uniform bound \eqref{eq:infB}, since for  $k>1$ it holds $H^k \subset L^\infty$. 

In \cite{AGM} the following estimate is obtained
\begin{equation}\label{eq:low:4}
\|u^\epsilon(t)\|_{\dot{H}^2}\leq\|u^\epsilon_0\|_{\dot{H}^2} e^{\int_0^t \|\nabla B\|_{L^\infty}ds}
\end{equation}
for the case where $B$ verifies \eqref{f:B}. A similar estimate holds also when $B$ solves \eqref{f:B1}. Recalling the classical inequality (see formula (3.2b) and Appendix 1 in \cite{kato1986well})
$$
\|\mathcal{T}f\|_{L^\infty}\leq c(\|f\|_{L^p}+\|f\|_{L^\infty}\max\{1,\log(\|f\|_{W^{s-1,p}}/\|f\|_{L^\infty}))\},\quad 1<p<\infty, \; s>1+2/p
$$
where $\mathcal{T}$ is a singular integral operator of Calder\'on-Zygmund type and using that $\nabla B$ is a zeroth order singular integral operator, we have that
$$
\|\nabla B\|_{L^\infty}\leq c(\|u^\epsilon\|_{L^2}+\|u^\epsilon\|_{L^\infty}\max\{1,\log(c\|u^\epsilon \|_{\dot{H}^{2}})\}),
$$
since, using mass conservation, $\|f\|_{H^{2}}/\|f\|_{L^\infty} \le 4 \pi^2 \|f\|_{\dot H^{2}}/\|f\|_{L^1}  +c\le c \|f\|_{H^{2}} +c$. Using \eqref{eq:infB} to the r.h.s. above and then plugging the resulting estimate into \eqref{eq:low:4}, we obtain the uniform estimate 
$$
\|u^\epsilon(t)\|_{\dot{H}^2}\leq H(T, \|u_0\|_{L^1},\|u_0\|_{L^\infty}, \|u_0\|_{\dot{H}^2},\alpha,r,\chi),
$$
where $H$ is an explicit function depending on $F$ in \eqref{eq:infB}. For the case $k=3$ we can repeat the same ideas and obtain that
$$
\|u^\epsilon(t)\|_{\dot{H}^3}\leq K(T\|u_0\|_{L^1},\|u_0\|_{L^\infty}, \|u_0\|_{\dot{H}^2},\|u_0\|_{\dot{H}^3},\alpha,r,\chi),
$$
where $K$ is another explicit function depending on $F$ in \eqref{eq:infB} and $H$.

Summing up, we can close estimate \eqref{eq:low:4} $\epsilon-$uniformly, hence (for an arbitrary $T < \infty$) obtaining the respective $\epsilon-$uniform estimates in 
$$
u^\epsilon\in L^\infty(0,T;H^k(\TT))\cap L^2(0,T;H^{k+\alpha/2}(\TT)),\quad k = 2,3
$$
Then we can pass to the limit and obtain
$$
u\in L^\infty(0,T;H^k(\TT))\cap L^2(0,T;H^{k+\alpha/2}(\TT)),\quad k = 2,3
$$
that solves  \eqref{eqPKS} in a weak sense. Moreover, using the structure of the Gronwall's inequality \eqref{eq:low:4} to obtain the continuity at time $t=0$ and the gain of regularity to obtain the continuity at later times $t>0$, we have that
\[
u\in C([0,T);H^k (\TT)).
\]
\section*{Acknowledgment}
JB is partially supported by the internal IMPAN grant for young researchers.
RGB is supported by the Labex MILYON and the Grant MTM2014-59488-P from the Ministerio de Econom\'ia y Competitividad (MINECO, Spain).
Part of the research leading to results presented here was conducted during a short stay of RGB at IMPAN within WCMCS KNOW framework.

\bibliographystyle{abbrv}

\end{document}